\documentclass{article}

\usepackage[utf8]{inputenc}
\usepackage{fullpage}
\usepackage{amssymb}
\usepackage{amsmath}
\usepackage{graphicx}
\usepackage{wrapfig}
\usepackage{cite}
\usepackage{subfigure}
\usepackage{makeidx}
\usepackage{multicol}
\usepackage{mathptmx}
\usepackage{mathrsfs}
\usepackage{mathtools}
\usepackage{amsthm}

\newtheorem{definition}{Definition}
\newtheorem{remark}{Remark}

\newtheorem{proposition}{Proposition}

\newtheorem{theorem}{Theorem}
\newtheorem{lemma}{Lemma}

\title{\textbf{On Separation Between Learning and Control in Partially Observed Markov Decision Processes}}
\author{Andreas A. Malikopoulos\\ Terri Connor Kelly and John Kelly Career Development Professor\\ University of Delaware}
\date{}

\begin{document}

\maketitle

\begin{center}
	\textbf{Abstract}
\end{center}
Cyber-physical systems (CPS) encounter a large volume of data which is added to the system gradually in real time and not altogether in advance. As the volume of data  increases, the domain of the control strategies also increases, and thus it becomes challenging to search for an optimal strategy. Even if an optimal control strategy is found, implementing such strategies with increasing domains is burdensome. To derive an optimal control strategy in CPS, we typically assume an ideal model of the system. Such  model-based control approaches cannot effectively facilitate optimal solutions with performance guarantees due to the discrepancy between the model and the actual CPS. Alternatively, traditional supervised learning approaches cannot always  facilitate robust solutions using data derived offline. Similarly, applying reinforcement learning approaches directly to the actual CPS might impose significant implications on safety and robust operation of the system. The  goal of this chapter is to provide a theoretical framework that  aims at separating the control and learning tasks which allows us to combine offline model-based control with online learning approaches, and thus circumvent the challenges in deriving optimal control strategies for CPS.

\section{Introduction}\label{intro}
Cyber-physical systems (CPS) encounter a large volume of data which is added to the system gradually in real time and not altogether in advance as for example in emerging mobility systems  \cite{zhao2019enhanced}, networked control systems \cite{Hespanha:2007aa,Zhang:2020aa}, mobility markets \cite{chremos2020SharedMobility,chremos2020MobilityMarket}, smart power grids \cite{Khaitan:2013aa,Howlader:2014aa}, power systems \cite{Du:2017aa}, cooperative cyber-physical networks \cite{Pasqualetti:2015aa,Sami:2016aa,Clark:2017aa}, social media platforms \cite{Dave2020SocialMedia}, cooperation of robots \cite{Jadbabaie:2003aa,Saulnier:2017aa,Beaver2020AnFlockingb}, and internet of things \cite{Li:2016aa,Xu:2018aa,Ansere:2020aa}.
In such  applications, we typically use a model to derive the optimal control strategy of the system. However, model-based control approaches cannot effectively facilitate optimal solutions due to the discrepancy between the model and the actual system. On the other hand, traditional supervised learning approaches cannot always  facilitate robust solutions using data derived offline. Similarly, applying reinforcement learning approaches directly to the actual CPS might impose undesired implications on safety and robust operation of the system. The overarching goal of this chapter is to investigate how to circumvent these challenges at the intersection of learning and control. 

The evolution of the state of many CPS, in several instances,  can be appropriately represented by a  Markov decision process (MDP) or partially observed MDP (POMDP). Reinforcement learning (RL) \cite{Bertsekas1996,Sutton1998a} has  been  widely used as an adaptive approach \cite{Narendra1989StableAS,Sastry1989AdaptiveCS,strm1989AdaptiveC,Ioannou2012RobustAC}  to derive the optimal control strategy in MDPs and POMDPs \cite{Sutton:1992ub} where a model of the system might not be available \cite{Dydek2013AdaptiveCO,Leman2009L1AC}.
In particular, research efforts have focused on addressing MDPs and POMDPs either with direct or indirect RL methods including  robust learning-based approaches \cite{Aswani:2013ue,Bouffard:2012wp}, learning-based model predictive control \cite{Hewing:2020uv,Rosolia:2018wv,Zhang:2020wf}, optimization of powertrain operation of vehicles with respect to the driver's driving style \cite{Malikopoulos2010a,Malikopoulos2011}, planning of autonomous vehicles \cite{You:2019va}, traffic control in simulation and scaled experiments \cite{Wu2017FlowAA,Wu:2017uz,Vinitsky:2018vx,jang2019simulation,chalaki2020ICCA,chalaki2020hysteretic}, decentralized learning for stochastic games \cite{Arslan:2017vo}, optimal social routing  \cite{Krichene:2018we}, congestion games \cite{Krichene:2015vx}, and  enhanced security against replay attacks in CPS
\cite{Zhai:2021wy,Sahoo:2020tx}. 

Implications on robustness of optimal control strategies derived using an ``incorrect" model and applied to the actual system have been discussed in \cite{Kara:2018vu}. Other efforts have focused on approximate planning and learning in POMDPs using an information state \cite{Subramanian2020ApproximateIS}. This approach provides a constructive way for RL in partially observed systems. More recent efforts have also combined model reference adaptive control with RL to generate online policies \cite{Guha2021OnlinePF}. 
Two recent survey papers \cite{Kiumarsi:2018tq,Recht2018ATO} provide a comprehensive review of the general RL problem formulations along with a complete list of applications.

In this chapter, we present a framework in which we aim at finding sufficient statistics to compress the growing data of the system without loss of optimality using a conditional probability of the state of the system at time $t\in\mathbb{R}_{\ge0}$ given all data available up until $t$. This conditional probability is called \textit{information state} of the system and takes values in a time-invariant space. We  use this information state to derive \textit{separated control strategies.} Separated control strategies are related to the concept of separation between the estimation of the information  state and control of the system. An important consequence of this separation is that for any given choice of control strategies and a realization of the system's variables until time $t$, the information states of the system at future times do not depend on the choice of the control strategy at time $t$ but only on the realization of the control at time $t$ \cite{Kumar1986}. Thus, the future information states are separated from the choice of the current control strategy. By establishing separated control strategies, we can derive offline the optimal control strategy of the system with respect to the information state, which takes values in a time-invariant space, and then use standard learning techniques  \cite{Brand:1999aa,Gyorfi:2007aa} to learn the information state online while data are added gradually to the system. This approach could effectively facilitate optimal solutions with performance guarantees in a wide range of complex CPS \cite{Malikopoulos2016c}.

The structure of this chapter is organized as follows. In Section \ref{sec:2}, we present the modeling framework and  formulation of the optimal control problem. In Section \ref{sec:3}, we provide results on separated control strategies. In Section \ref{sec:4}, we illustrate the  framework with a simple example. Finally, we provide concluding remarks and discuss potential directions for future research in Section \ref{sec:5}.

\subsection{Notation}
We denote random variables with upper case letters, and their realizations with lower case letters, e.g., for a random variable $X_t$, $x_t$ denotes its realization. Subscripts denote time. The expectation of a random variable is denoted by $\mathbb{E}[\cdot]$, the probability of an event is denoted by $\mathbb{P}(\cdot)$, and the probability density function is denoted by $p(\cdot)$. 
For a control strategy $\bf{g}$, we use $\mathbb{E}^{\bf{g}}[\cdot]$, $\mathbb{P}^{\bf{g}}(\cdot)$, and $p^{\bf{g}}(\cdot)$ to denote that the expectation, probability, and probability density  function, respectively, depend on the choice of the control strategy $\bf{g}$. For two measurable spaces $(\mathcal{X}, \mathscr{X})$ and $(\mathcal{Y}, \mathscr{Y})$, $\mathscr{X}\otimes\mathscr{Y}$ is the product $\sigma$-algebra on $\mathcal{X}\times \mathcal{Y}$ generated  by the collection of all measurable rectangles, i.e., $\mathscr{X}\otimes\mathscr{Y}\coloneqq \sigma(\{A\times B: A\in\mathscr{X}, B\in\mathscr{Y} \})$. The product of  $(\mathcal{X}, \mathscr{X})$ and $(\mathcal{Y}, \mathscr{Y})$ is the measurable space $(\mathcal{X}\times \mathcal{Y}, \mathscr{X}\otimes\mathscr{Y})$. 

\section{Separated Control Strategies}
\label{sec:2}
We  consider a system evolving as a POMDP  in which there is  a large volume of data that is added to the system gradually and not altogether in advance. 
We seek to separate the control and learning tasks in the system which can eventually help us combine offline model-based control with online learning approaches.  In particular, we aim at finding sufficient statistics to compress the growing data of the system without loss of optimality  using a conditional probability of the state of the system at time $t$ given all the data available up until time $t.$ This conditional probability is called information state and it takes values in a time-invariant space. Using this information state, we can derive results for optimal control strategies in a time-invariant domain. Results based on data which, even though they increase with time, are compressed to a sufficient statistic taking values in a time-invariant space are called \textit{structural results} (see \cite{Krishnamurthy2016}, p. 203). Structural results can help us establish separated control strategies \cite{Malikopoulos2021}, and thus they are related to the concept of separation between estimation and control. An important consequence of this separation is that for any given choice of control strategies and a realization of the system’s variables until time $t,$ the information states at future times do not depend on the choice of the control strategy at time $t$ but only on the realization of the decision at time $t$ (see \cite{Kumar1986}, p. 81). 
Thus, the future information states are separated from the choice of the current control strategy. The latter is necessary in order to formulate a classical dynamic program, where at each step the optimization problem is to find the optimal decision for a given realization of the information state \cite{Howard,Bertsekas2017}. 
By establishing separated control strategies, we can derive offline the optimal control strategy of the system with respect to the information state, which might not be precisely known due to model uncertainties or complexity of the system, and then use  standard learning techniques  to learn the information state online while data are added gradually to the system in real time. Structural results can also help us derive optimal strategies in decentralized systems \cite{Mahajan2016,Nayyar:2014aa,Nayyar2013b,Dave2020a,Dave2021a}.

\begin{figure}
	\centering
	\includegraphics[width=0.6\linewidth, keepaspectratio]{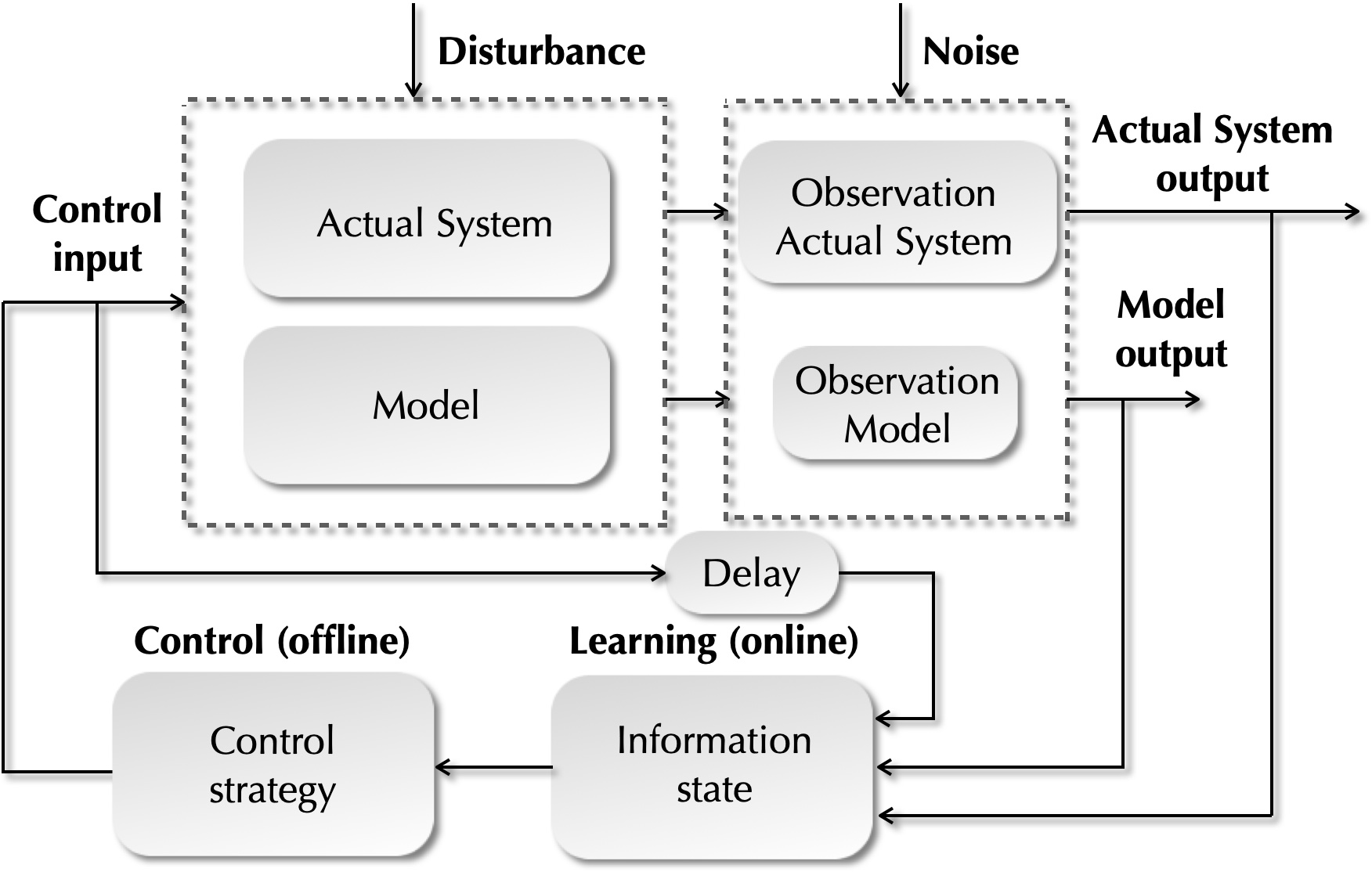} 
	\caption{Illustration of the proposed framework.}%
	\label{fig:1}%
\end{figure}

More specifically, in our framework illustrated in Fig. \ref{fig:1}, we use the actual system that we seek to optimally control online, in parallel with a  model of the system that we have available. We establish an information state which is the conditional joint probability distribution of the states of the model and the actual system at time $t$ given all data available of the model up until time $t$, i.e., $p(\text{state of model, state of actual system}~ |$ $~ \text{data of the model})$. 
Then, we use this information state in conjunction with the model to derive offline separated control strategies.
Since we derive the optimal strategies offline, the state of the actual system  is not known, i.e., the actual system operates only online, and thus the optimal strategy of the model is parameterized with respect to all realizations of the  state of the actual system. However,  since the control strategy and the process of estimating the information state are separated,  we can learn the information state of the system online, while we operate simultaneously the model and the actual system in real time. Namely,  the optimal strategy derived for the  model offline, which is parameterized with respect to the state of the actual system, is used to operate the actual system in parallel with the model. As we operate both the actual system and the model and collect data, we can learn the information state online.
In our exposition, we show that when the information state becomes known online through learning, the separated control strategy of the model derived offline is optimal for the actual system.

\subsection{An Illustrative CPS Application: Separation Between Learning and Control }

In this section, we outline how we could potentially separate the learning and control tasks in a CPS application.
Consider a number of connected and automated vehicles (CAVs) that need to coordinate in a given traffic scenario, e.g., crossing a signal-free intersection (Fig. \ref{fig:3}), in which a large volume of data by CAVs and infrastructure is produced  gradually in real time. 
CAVs are typical CPS where the cyber component (data and shared information through vehicle-to-vehicle and vehicle-to-infrastructure communication) can aim at optimally controlling the physical entities (CAVs, non-CAVs). 

\begin{figure}
	\centering
	\includegraphics[width=0.6\linewidth, keepaspectratio]{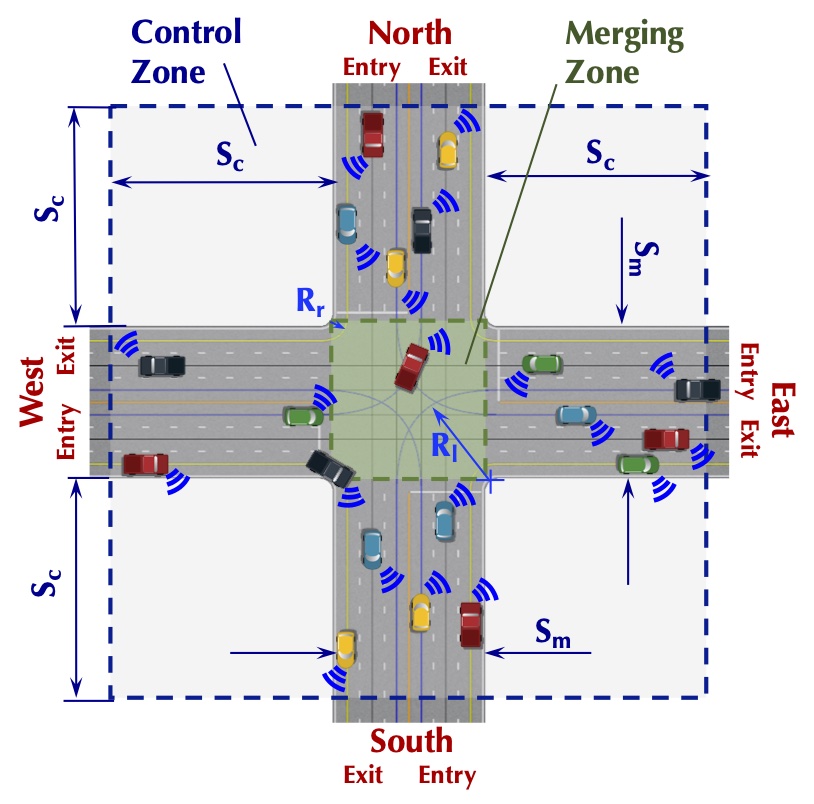} 
	\caption{A signal-free intersection with connected and automated vehicles.}%
	\label{fig:3}
\end{figure}

The region at the center of the intersection, called \textit{merging zone}, is the area of potential lateral collision of the vehicles. The intersection has also a \textit{control zone} (Fig. \ref{fig:3}) inside of which the CAVs can communicate with each other and the infrastructure to exchange information, e.g., share their position, speed, occupied lane, and route within the control zone. The objective is to derive a control strategy for the CAVs to cross the intersection by minimizing fuel consumption without the use of traffic lights, without creating congestion, and under the hard safety constraint of collision avoidance. 
The actual system consists of the individual CAVs inside the control zone. The state of the system  is the position and speed of CAVs inside the control zone, the control input is the acceleration/deceleration of each CAV, and the cost function is the fuel consumption of all CAVs.

To compute the optimal control input for each CAV, we typically model each CAV as a double integrator and consider perfect communication among the CAVs, e.g., without any delays \cite{Malikopoulos2020}. This allows us to derive a closed-form analytical solution \cite{malikopoulos2019ACC}, which exists under certain conditions \cite{mahbub2020Automatica-2}, that yields for each CAV the optimal control input (acceleration/deceleration) at any time in the sense of minimizing fuel consumption. However, implementing this solution in physical CAVs (the actual system) imposes significant implications due the discrepancy between the double integrator model and the dynamics of a physical CAV in addition to the presence of communication delays \cite{Zhao2018ITSC}. First, due to the existing delays in the communication among the CAVs,  the precise position and speed of each CAV inside the control zone is not known to other CAVs. Therefore, the  position and speed trajectories of each CAV  resulting from the solution using the double integrator models might activate state, control, and safety constraints within the control and merging zones. Second, depending on their size and weight, some CAVs might not be able to follow the optimal control input (acceleration/deceleration) given by the solution since the optimal control input is derived using a double integrator model which is far from being able to capture the dynamics of a vehicle. Thus, the actual position and speed trajectories of each CAV, and as a result, the true state of the system (actual system) will be different from what it is expected by the solution. Learning the policy online might impose undesirable implications on safety and robust operation of the system.

To separate learning and control in this application, the first step is to establish an information state which is the conditional joint probability distribution of the speed and position of the CAVs resulted by the double integrator models and the actual ones resulted by the physical CAVs at time $t$ given all data available by the double integrator models up until time $t$, i.e., $p(\text{state of CAVs by double integrator models, state of}$ $\text{physical CAVs data of the}~ |$ $~ \text{ double integrator models})$.  
Then, we use this information state along with the double integrator models to derive offline the optimal control input (acceleration/deceleration profile) of the CAVs.
Since we derive the optimal control input of each CAV offline, the state of the actual system (position and  speed of the physical CAVs) is not known, so the control input is parameterized with respect to different possible realizations of the speed and position of the physical CAVs. The next step is to implement the parameterized control input derived offline in the physical CAVs. 
This step is essentially the same as the one of the original approach described above with the only difference being that the control input now is parameterized. However, as we operate simultaneously the physical CAVs and the double integrator models, we collect data and  learn the information state of the system online. 
When the information state becomes known online through learning, the control input derived offline using the double integrator models is optimal for the physical CAVs.

\subsection{Modeling Framework} 
\label{sec:2a}

We  consider a system with a measurable state space $(\mathcal{X}_t, \mathscr{X}_t)$, where $\mathcal{X}_t$ is the set in which the system state takes values at time $t = 0,1,\ldots,T$, $T\in\mathbb{N}$, and $\mathscr{X}_t$ is the associated $\sigma$-algebra. Let $X_t$ be a random variable that represents the state of the model of the system and $\hat{X}_t$ be a random variable that represents the state of the actual system. Both  random variables are defined on the probability space $(\Omega, \mathscr{F}, \mathbb{P})$, i.e., $X_t: (\Omega, \mathscr{F})\to(\mathcal{X}_t, \mathscr{X}_t)$,  $\hat{X}_t: (\Omega, \mathscr{F})\to(\mathcal{X}_t, \mathscr{X}_t)$, where $\Omega$ is the sample space, $\mathscr{F}$ is the associated $\sigma$-algebra, and $\mathbb{P}$ is a probability measure  on $(\Omega, \mathscr{F})$. The control of the actual system is represented by a random variable $U_t: (\Omega, \mathscr{F})\to(\mathcal{U}_t, \mathscr{U}_t),$ defined on the probability space $(\Omega, \mathscr{F}, \mathbb{P})$, and takes values in the measurable space $(\mathcal{U}_t, \mathscr{U}_t)$,  where $\mathcal{U}_t$ is  the system's nonempty feasible set of actions at time $t$ and $\mathscr{U}_t$ is the associated $\sigma$-algebra. The actual system and its corresponding model are illustrated in Fig. \ref{fig:2}.

\begin{figure}
	\centering
	\includegraphics[width=0.7\linewidth, keepaspectratio]{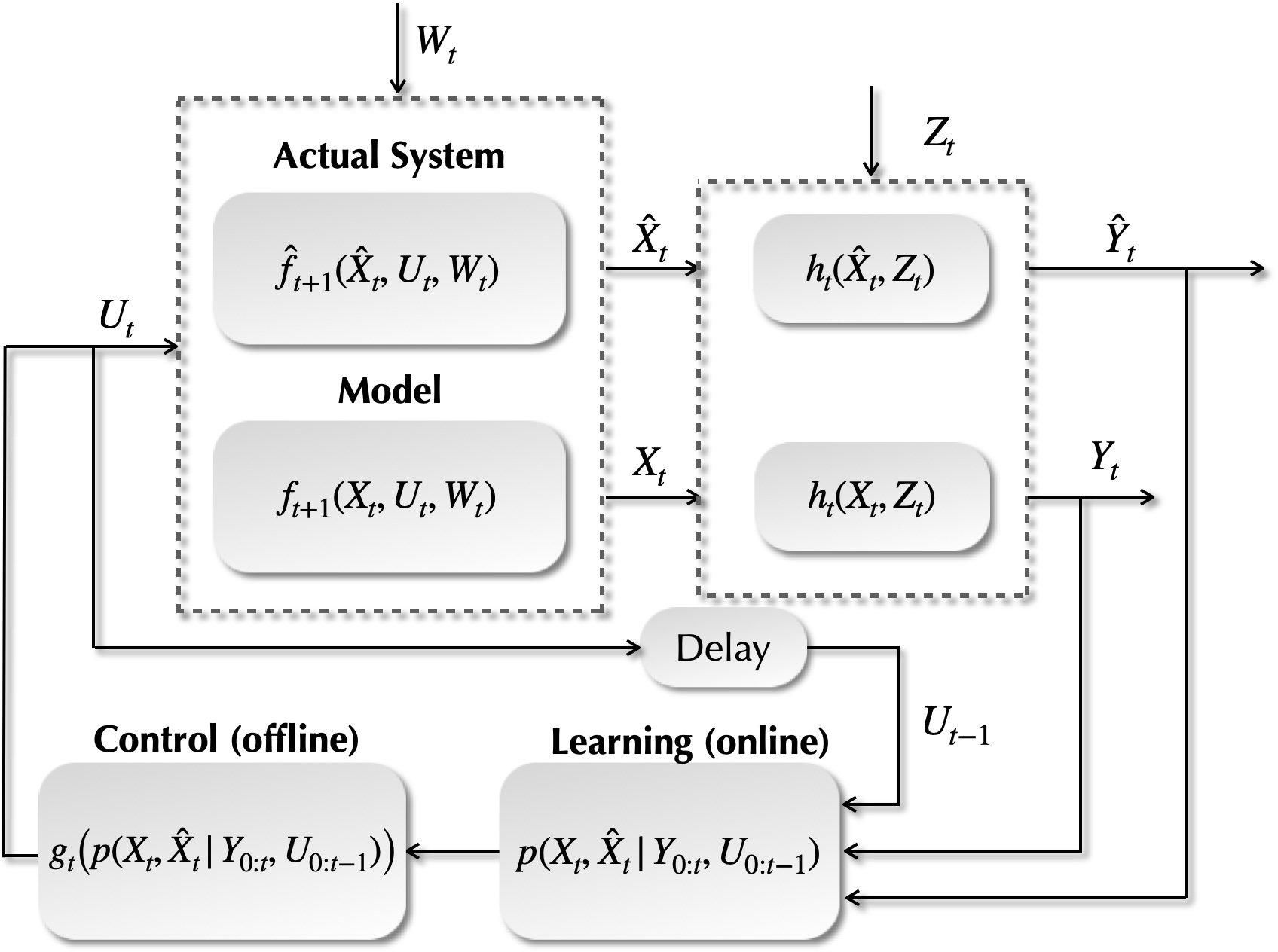} 
	\caption{Separation of learning and control.}%
	\label{fig:2}%
\end{figure}

Let ${U}_t$ be the control of the actual system at time $t$.  Starting at the initial state $X_0$, the evolution of the mathematical model of the system that we have available is described by the state equation
\begin{align}\label{eq:state}
	X_{t+1}=f_t\big(X_t,U_t,W_t\big), 
\end{align}
where  $t = 0,1,\ldots,T-1$, and $W_t$ is a random variable defined on the probability space $(\Omega, \mathscr{F}, \mathbb{P})$ that corresponds to the external, uncontrollable disturbance to the actual system, hence to the model too, and takes values in a measurable set $(\mathcal{W}, \mathscr{W})$, i.e., $W_t:(\Omega, \mathscr{F})\to(\mathcal{W}, \mathscr{W})$. 
Similarly, starting at the initial state $X_0$, the evolution of the actual system  is described by the state equation
\begin{align}\label{eq:statereal}
	\hat{X}_{t+1}=\hat{f}_t\big(\hat{X}_t,U_t,W_t\big),
\end{align}
where $t = 0,1,\ldots,T-1$, while $\{W_t: t=0,\ldots,T-1\}$ is a sequence of independent random variables that are also independent of the initial state $X_0$.

At time $t = 0,1,\ldots,T-1$, we make an observation $Y_t$ of the model's output, which takes values in a measurable set $(\mathcal{Y}, \mathscr{Y})$, described by the observation equation 
\begin{align}\label{eq:observe}
	Y_t = h_t(X_t,Z_t),
\end{align}
where $Z_t$ is a random variable defined on the probability space $(\Omega, \mathscr{F}, \mathbb{P})$ that corresponds to the noise of the sensor and takes values in a measurable set $(\mathcal{Z}, \mathscr{Z})$, i.e., $Z_t:(\Omega, \mathscr{F})\to(\mathcal{Z}, \mathscr{Z})$. Note  $\{Z_t: ~t=0,\ldots,T-1 \}$ is a sequence of independent random variables that are also independent of $\{W_t: t=0,\ldots,T-1\}$, and the initial state $X_0$. Similarly, at time $t = 0,1,\ldots,T-1$, we make an observation $\hat{Y}_t$ of the actual system, which takes values in a measurable set $(\mathcal{Y}, \mathscr{Y})$, described by the observation equation 
\begin{align}\label{eq:observereal}
	\hat{Y}_t = h_t(\hat{X}_t, Z_t).
\end{align}

A control strategy $\textbf{g}=\{g_t;~ t=0,\ldots,T-1\}$, $\textbf{g}\in\mathcal{G}$, where $\mathcal{G}$ is the feasible set of strategies, yields a decision 
\begin{align}\label{eq:control}
	U_t =g_t(\hat{Y}_{0:t}, U_{0:t-1}),
\end{align}
where  the measurable function $g_t$ is the control law.

\textbf{Problem 1}~[Actual system]: \label{problem1}
The problem is to derive the optimal control strategy $\textbf{g}^*\in\mathcal{G}$ that minimizes the expected total cost of the actual system, 
\begin{align}\label{eq:cost}
	\hat{J}(\textbf{g})=\mathbb{E}^{\textbf{g}}\left[\sum_{t=0}^{T-1} c_t(\hat{X}_t, U_t)+c_T(\hat{X}_T)\right],
\end{align}
where the expectation is with respect to the joint probability distribution of the random variables $\hat{X}_t$ and  $U_t$ designated by the choice of $\textbf{g}\in\mathcal{G}$, $c_t(\cdot, \cdot): \mathcal{X}_t\times \mathcal{U}_t \to\mathbb{R}$ 
is the measurable cost function of the actual CPS at $t$, and $c_T(\cdot):(\mathcal{X}_T, \mathscr{X}_T) \to\mathbb{R}$ is the measurable cost function at $T$.

The statistics of the primitive random variables $X_0$,  $\{W_t: t=0,\ldots,T-1\}$, $\{Z_t :~ t=0,\ldots,T-1\}$, the observation equations $\{h_t: ~ t=0,\ldots,T-1\}$, and the cost functions $\{c_t: t=0,\ldots,T\}$ are all known. However, the state equations $\{\hat{f}_t: t=0,\ldots,T-1\}$ are not known.


\section{Separation of Learning and Control } \label{sec:3}
In our exposition, we address Problem $1$ from the point of view of a central controller who seeks to derive the optimal strategy $\textbf{g}\in\mathcal{G}$ of the actual system. We consider densities for all probability distributions to simplify notation. Let $\textbf{g}=\{g_t;~ t=0,\ldots,T-1\}$, $\textbf{g}\in\mathcal{G},$ be a control strategy which yields a decision $U_t =g_t(Y_{0:t}, U_{0:t-1})$. 
First, we institute an appropriate information state that can be used to formulate a classical dynamic programming decomposition. To establish this information state, we use the model of the system in conjunction with the actual system (Fig. \ref{fig:2}). 

The information state, denoted by $\Pi_{t}(Y_{0:t}, U_{0:t-1})(X_{t},\hat{X}_t)$ and defined formally next,  is the probability density function $p(X_{t}, \hat{X}_t ~|~Y_{0:t}, U_{0:t-1})$. In what follows, to simplify notation, the information state $\Pi_{t}(Y_{0:t}, U_{0:t-1})$ $(X_{t},\hat{X}_{t})$ at $t$ is denoted simply by $\Pi_t$. We use its arguments only if it is required in our exposition.

\begin{definition} \label{def:infoteam}
	An information state, $\Pi_t$, for the system illustrated in Fig. \ref{fig:2} described by the state equations \eqref{eq:state} and \eqref{eq:statereal}, is (a) a function of  $(Y_{0:t}, U_{0:t-1})$, while (b) $\Pi_{t+1}$ is determined from $\Pi_t$, $Y_{t+1}$, and $U_{t}$.
\end{definition}

Next, we provide some necessary results that aim at establishing the information state. The results of the following Lemmas \ref{lem:y_t} -- \ref{lem:x_t} are equivalent to the results of \cite[Lemmas 1-3]{Malikopoulos2022a} when the information structure  of the system is classical \cite{Schuppen2015,Yuksel2013} and the controller has perfect recall \cite{Kumar1986,bertsekas1995dynamic}.

\begin{lemma} \label{lem:y_t}
	For any control strategy $\textbf{g}\in\mathcal{G}$ of the system, the conditional probability of the model's observation $Y_{t+1}$ at $t+1$ given the states of the model and the actual system $X_{t+1}$ and $\hat{X}_{t+1}$ at $t+1$, respectively,  the history of the model's observations $Y_{0:t}$, and the history of the  control actions $U_{0:t}$ is equal to the conditional probability of the model's observation $Y_{t+1}$ at $t+1$ given the states of the model $X_{t+1}$ at $t+1$, which  does not depend on the control strategy $\textbf{g}$. Equivalently, we have
	\begin{align}\label{eq:y_t}
		p^{\textbf{g}}(Y_{t+1}~|~X_{t+1}, \hat{X}_{t+1}, Y_{0:t}, U_{0:t})= p(Y_{t+1}~|~X_{t+1}),
	\end{align}
	for all $t=0,1,\ldots, T-1.$
\end{lemma}

\begin{lemma} \label{lem:x_t1ut}
	For any control strategy $\textbf{g}\in\mathcal{G}$ of the system, the conditional probability of the joint distribution of the states of the model and the actual system $X_{t+1}$ and $\hat{X}_{t+1}$ at $t+1$, respectively, given the states of the model and the actual system $X_{t}$ and $\hat{X}_{t}$ at $t$, respectively, the history of the model's observations $Y_{0:t}$, and the history of the  control actions $U_{0:t}$ is equal to the conditional probability of the joint distribution of the states of the model and the actual system $X_{t+1}$ and $\hat{X}_{t+1}$ at $t+1$, respectively, given the states of the model and the actual system $X_{t}$ and $\hat{X}_{t}$, respectively, and the control action $U_t$ at $t$,  which  does not depend on the control strategy $\textbf{g}$. Equivalently, we have
	\begin{align}
		p^{\textbf{g}}(X_{t+1},\hat{X}_{t+1}~|~X_t, \hat{X}_{t},  Y_{0:t}, U_{0:t}) 
		= p(X_{t+1}, \hat{X}_{t+1}~|~X_t, \hat{X}_{t}, U_t), \label{eq:x_t1ut}
	\end{align}
	for all $t=0,1,\ldots, T-1.$
\end{lemma}

\begin{lemma} \label{lem:x_t}
	For any control strategy $\textbf{g}\in\mathcal{G}$ of the system, the conditional probability of the joint distribution of the states of the model and the actual system $X_{t}$ and $\hat{X}_{t}$ at $t$, respectively, given the history of the model's observations $Y_{0:t}$, and the history of the  control actions $U_{0:t-1}$ does not depend on the control strategy $\textbf{g}$. Equivalently, we have
	\begin{align}\label{eq:x_t}
		p^{\textbf{g}}(X_{t}, \hat{X}_{t}~|~Y_{0:t}, U_{0:t-1}) = p(X_{t},\hat{X}_{t}~|~Y_{0:t}, U_{0:t-1}),
	\end{align}
	for all $t=0,1,\ldots, T-1.$ 
\end{lemma}

\begin{remark}\label{cor:lemU}
	As a consequence of Lemma \ref{lem:x_t}, and since both $X_{t}$ and $\hat{X}_{t}$ do not depend on $U_t$, for any control strategy $\textbf{g}\in\mathcal{G}$ of the system, the conditional probability of the joint distribution of the states of the model and the actual system $X_{t}$ and $\hat{X}_{t}$ at $t$, respectively, given the history of the model's observations $Y_{0:t}$, and the history of the  control actions $U_{0:t}$ up until $t$ is equal to the conditional probability of the joint distribution of the states of the model and the actual system $X_{t}$ and $\hat{X}_{t}$ at $t$, respectively, given the history of the model's observations $Y_{0:t}$, and the history of the  control actions $U_{0:t-1}$ up until $t-1$	which does not depend on the control strategy $\textbf{g}$. Equivalently, we have
	\begin{align}\label{eq:lemU}
		p^{\textbf{g}}(X_{t}, \hat{X}_t~|~Y_{0:t}, U_{0:t}) = p(X_{t}, \hat{X}_t~|~Y_{0:t}, U_{0:t-1}).
	\end{align}
\end{remark}

The next result shows that such information state does not depend on the control strategy of the model.

\begin{theorem}\label{theo:y_t}
	For any control strategy $\textbf{g}\in\mathcal{G}$ derived offline using the model of the system, the information state $\Pi_{t}(Y_{0:t}, U_{0:t-1})(X_{t},\hat{X}_t)$ does not depend on the control strategy $\textbf{g}$.
	Moreover, there is a function $\phi_t$, which does not depend on the control strategy $\textbf{g}$, such that
	\begin{align}\label{eq:xt1}
		\Pi_{t+1}(Y_{0:t+1}, U_{0:t})(X_{t+1},\hat{X}_{t+1}) 
		= \phi_t\big[ \Pi_{t}(Y_{0:t}, U_{0:t-1})(X_{t},\hat{X}_t), Y_{t+1}, U_t \big],
	\end{align}
	for all $t=0,1,\ldots, T-1.$
\end{theorem}

The result of Theorem \ref{theo:y_t} follows from Lemmas \ref{lem:y_t} -- \ref{lem:x_t1ut} and Remark \ref{cor:lemU}. Note that the system's information state $\Pi_{t+1}(Y_{0:t+1}, U_{0:t})(X_{t+1},$ $ \hat{X}_{t+1})$ is the entire probability density function and not just its value at any particular realization of  $X_{t+1}$ and $ \hat{X}_{t+1}$. This is because to compute $\Pi_{t+1}(Y_{0:t+1}, U_{0:t})(X_{t+1},\hat{X}_{t+1})$ for any particular realization of  $X_{t+1}$ and $\hat{X}_{t+1}$, we need the probability density functions $p(~\cdot, \cdot  ~|~ Y_{0:t}, U_{0:t})$ and $p(~\cdot, \cdot  ~|~ Y_{0:t}, U_{0:t-1})$. This implies that the information state takes values in the space of these probability densities, which is an infinite-dimensional space.


\begin{definition}\label{def:septeam}
	A control strategy $\textbf{g}=\{g_t;~ t=0,\ldots,T-1\}$, of the system is said to be \textit{separated} if $g_t$ depends on $Y_{0:t+1}$ and  $U_{0:t}$ only through the information state, i.e., $U_t  = g_t\big(\Pi_{t}(Y_{0:t}, U_{0:t-1})(X_{t},\hat{X}_{t})\big)$. Let $\mathcal{G}^s\subseteq\mathcal{G}$ denote the set of all separated control strategies.
\end{definition}

To derive the optimal control strategy of the actual system in Problem 1, we formulate the following optimization problem.

\textbf{Problem 2} \label{problem2}
Using the model of the actual system, we seek to derive offline the optimal control strategy $\textbf{g}^*\in\mathcal{G}^s$ that minimizes the following expected total cost 
\begin{align}			
	J(\textbf{g};\hat{x}_{0:T})
	= \mathbb{E}^{\textbf{g}}\Bigg[\sum_{t=0}^{T-1}\Big[c_t(X_t, U_t)
	+ \beta \cdot|X_{t+1}- \hat{X}_{t+1}|^2\Big]
	+c_T(X_T) \Bigg], \label{eq:costreal}	
\end{align}	
where $X_{t+1}=f_t\big(X_t,U_t,W_t\big)$, $\hat{X}_{t+1}=\hat{f}_t\big(\hat{X}_t,$ $U_t,W_t\big)$, and $\beta$ is a factor to adjust the units and size of the norm accordingly as designated by the cost function $c_t(\cdot, \cdot)$. The norm penalizes any discrepancy between the realizations of the state of the system's model and the state of the actual system. The expectation in \eqref{eq:costreal}	is with respect to the joint probability distribution of the random variables $X_t$,  $U_t$, $\hat{X}_t$, $t=0,1,\ldots, T,$ (designated by the choice of $\textbf{g}\in\mathcal{G}^s$) and $W_t$.
Since solving \eqref{eq:costreal}	is an offline process, the realizations $\hat{x}_{0:T}$ of the state $\hat{X}_{0:T}$ of the actual system  are not known, and thus $\textbf{g}^*$ is parameterized with respect to $\hat{x}_{0:T}$. The statistics of the primitive random variables $X_0$,  $\{W_t: t=0,\ldots,T-1\}$, $\{Z_t: ~ t=0,\ldots,T-1\}$, the state equations $\{f_t: t=0,\ldots,T-1\}$, the observation equations $\{h_t: ~ t=0,\ldots,T-1\}$, and the cost functions $\{c_t: t=0,\ldots,T\}$ are all known.

Next, we use the information state $\Pi_{t}(Y_{0:t}, U_{0:t-1})$ $(X_{t},\hat{X}_{t})$ to  derive offline the optimal separated control strategy in Problem 2. In our exposition, we define recursive functions, and show that a separated control strategy, namely, a control strategy $\textbf{g}=\{g_t;~ t=0,\ldots,T-1\}$ where $g_t$ depends on $Y_{0:t+1}$ and  $U_{0:t}$ only through the information state, i.e., $U_t  = g_t\big(\Pi_{t}(Y_{0:t}, U_{0:t-1})(X_{t},\hat{X}_{t})\big)$, of the  system's model is optimal. 
In addition, we obtain a classical dynamic programming decomposition. 

\begin{theorem} \label{theo:dp}
	Let $V_t\big(Y_{0:t}, U_{0:t-1})(X_{t},\hat{X}_{t});~ \hat{x}_{t}\big)$ be functions defined recursively for all $\textbf{g}\in\mathcal{G}^s$ by
	\begin{align}
		&V_T\big(\Pi_{T}(Y_{0:T}, U_{0:T-1})(X_{T},\hat{X}_{T})\big) \coloneqq \mathbb{E}^{\textbf{g}}\Big[c_T(X_T)~|~\Pi_{T}=\pi_T \Big], \\
		&V_t\big(\Pi_{t}(Y_{0:t}, U_{0:t-1})(X_{t},\hat{X}_{t});~ \hat{x}_{t}\big)\coloneqq \inf_{u_t\in\mathcal{U}_t }\mathbb{E}^{\textbf{g}}\Big[c_t(X_t, U_t) 
		+ \beta ~|X_{t+1}  - \hat{X}_{t+1}|^2\nonumber\\
		&+V_{t+1}\big(\phi_t\big[ \Pi_{t}(Y_{0:t}, U_{0:t-1})(X_{t},\hat{X}_{t}), Y_{t+1}, U_t\big];~ \hat{x}_{t+1}\big)~|~\Pi_{t}=\pi_t, 	U_t=u_t  \Big], \label{theo2:1b}		
	\end{align}
	where $c_T(X_T)$ is the cost function at $T$;  $\beta$ is a factor to adjust the units and size of the norm as designated by the cost function $c_t(\cdot, \cdot)$; and $\pi_T$, $\pi_t$, and $u^{1:K}_t$ are the realizations of $\Pi_{T}$, $\Pi_{t}$, and $U^{1:K}_t$, respectively.
	Then, (a) for any control strategy $\textbf{g}\in\mathcal{G}^s$,
	\begin{align}			
		&V_t\big(\Pi_{t}(Y_{0:T}, U_{0:T-1})(X_{t},\hat{X}_{t});~ \hat{x}_{t}\big) \le J_t(\textbf{g};\hat{x}_{t:T}) \coloneqq \mathbb{E}^{\textbf{g}}\Bigg[\sum_{l=t}^{T-1}\Big[c_l(X_l,U_l)\nonumber\\
		&+ \beta \cdot|X_{l+1} - \hat{X}_{l+1}|^2 \Big]+ c_T(X_T) ~|~Y_{0:T}, U_{0:T-1}  \Bigg],
		\label{theo2:1c}	
	\end{align}	
	where $J_t(\textbf{g};\hat{x}_{t:T})$ is the cost-to-go function of the system's model, parameterized by the realizations of the  state $\hat{X}_{t}$ of the actual system, at time $t$ corresponding to the control strategy $\textbf{g}$; and
	(b) $\textbf{g}\in\mathcal{G}^s$ is optimal and 
	\begin{align}			
		V_t\big(\Pi_{t}(Y_{0:T}, U_{0:T-1})(X_{t},\hat{X}_{t});~ \hat{x}_{t}\big)=J_t(\textbf{g};~\hat{x}_{t:T}), \label{theo3:1b}
	\end{align}	
	with probability $1$.
\end{theorem}

The result of Theorem \ref{theo:y_t} is equivalent to the result of \cite[Theorem 2]{Malikopoulos2022a} when the system's information structure  is classical \cite{Schuppen2015,Yuksel2013} and the controller has perfect recall \cite{Kumar1986,bertsekas1995dynamic}. 

The optimal strategy derived by the system's model is parameterized with respect to the potential realizations $\hat{x}_{0:T}$ of the state $\hat{X}_{t}$  of the actual system. Then, we use this strategy to operate the actual system in parallel with the system's model (Fig. \ref{fig:2}) and we collect data from both. Using these data, we learn the information state $\Pi_{t}(Y_{0:T}, U_{0:T-1})(X_{t+1},\hat{X}_{t+1})$  online.

\begin{proposition} \label{theo:CPSstate}
	The information state $\Pi_{t}(Y_{0:t}, U_{0:t-1})(X_{t},$ $\hat{X}_{t})$ of the system illustrated in Fig. \ref{fig:2} is a function of  $p(X_{t} ~|~ Y_{0:t}, U_{0:t-1})$, $p(\hat{X}_{t} ~|~ \hat{Y}_{0:t}, U_{0:t-1})$, and $p(\hat{Y}_{0:t}~|~U_{{0:t}-1})$.
\end{proposition}
\begin{proof}
	Recall $\Pi_{t}(Y_{0:t}, U_{0:t-1})(X_{t},\hat{X}_t) =p(X_{t},\hat{X}_t$ $~|~Y_{0:t}, U_{0:t-1})$. Next,
	\begin{align}	
		&p(X_{t},\hat{X}_t~|~Y_{0:t}, U_{0:t-1})\nonumber\\
		&=\frac{p(\hat{X}_t~|~ X_{t},Y_{0:t}, U_{0:t-1})\cdot p(X_t, Y_{0:t}, U_{0:t-1})}{p(Y_{0:t}, U_{0:t-1})} \nonumber\\
		&=\frac{p(\hat{X}_t~|~ U_{0:t-1})\cdot p(X_t, Y_{0:t}, U_{0:t-1})}{p(Y_{0:t}, U_{0:t-1})} \nonumber\\
		&= p(\hat{X}_t~|~ U_{0:t-1})\cdot p(X_{t} ~|~ Y_{0:t}, U_{0:t-1}),
		\label{theoCPSstate:1b}	
	\end{align}	
	where, in the second equality, we used the fact that $\hat{X}_t$ does not depend on $X_t$ and $Y_{0:t}$, and in the third equality we applied Bayes' rule. The first term in \eqref{theoCPSstate:1b}	 can be written as
	\begin{align}
		p(\hat{X}_t~|~ U_{0:t-1})=\int_{\mathscr{X}_{t}} p(\hat{X}_t~|~\hat{Y}_{0:t}, U_{0:t-1})\cdot p(\hat{Y}_{0:t}~|~ U_{0:t-1}) d\hat{Y}_{0:t}.
		\label{theoCPSstate:1c}	
	\end{align}
	Substituting \eqref{theoCPSstate:1c} into \eqref{theoCPSstate:1b}, the result follows.
\end{proof}

\begin{remark} \label{rem:infostate}
	The conditional probabilities $p(X_{t}~|~Y_{0:t}, U_{0:t-1})$ and $p(\hat{X}_{t}~|~\hat{Y}_{0:t}, \hat{U}_{0:t-1})$ can be computed from the following recursive equations starting from the initial priors $p(X_{0}~|~Y_{0}, U_{0})$ and $p(\hat{X}_{0}~|~\hat{Y}_{0}, \hat{U}_{0})$, respectively,
	\begin{align}
		p(X_{t}~|~Y_{0:t}, U_{0:t-1})) &= \theta_{t-1}\big[ p(X_{t-1}~|~Y_{0:t-1}, U_{0:t-2}), Y_{t}, U_{t-1} \big],\\
		p(\hat{X}_{t}~|~\hat{Y}_{0:t}, U_{0:t-1}) &= \hat{\theta}_{t-1}\big[ p(\hat{X}_{t-1}~|~\hat{Y}_{0:t-1}, U_{0:t-2}), \hat{Y}_{0:t}, U_{t-1} \big],
	\end{align}
	for all $t=0,1,\ldots, T-1,$ where $\theta_{t}$ and $\hat{\theta}_{t}$ are appropriate functions \cite{Malikopoulos2021}.
\end{remark}

\begin{remark} \label{rem:CPSstate}
	The information state $\Pi_{t}(Y_{0:t}, U_{0:t-1})(X_{t},\hat{X}_{t})$ of the system illustrated in Fig. \ref{fig:2} can be obtained by using standard approaches, i.e., \cite{Brand:1999aa,Gyorfi:2007aa}, to learn online the conditional probabilities $p(\hat{Y}_{0:t}~|~U_{{0:t}-1})$ while we operate the actual system. 
\end{remark}

Next, we show that after the information state becomes known through learning, then the separated control strategy of the system's model derived offline is optimal for the actual system.

\begin{theorem} \label{theo:CPSmodel}
	Let $\textbf{g}\in\mathcal{G}^s$ be an optimal separated control strategy derived offline for the system's model which minimizes the expected total cost,
	\begin{align}
		J(\textbf{g};\hat{x}_{0:T})\coloneqq \mathbb{E}^{\textbf{g}}\Bigg[\sum_{t=0}^{T-1}\Big[ c_t(X_t,U_t)+ \beta\cdot|X_{t+1} - \hat{X}_{t+1}|^2\Big] + c_T(X_T)  \Bigg],
		\label{theo:CPSmodelgo}	
	\end{align}	
	in Problem 2. If $p(X_{t},\hat{X}_{t}~|~Y_{0:t}, U_{0:t-1})$ is known, then $\textbf{g}$ minimizes also the expected total cost of the actual system, 
	\begin{align}\label{eq:CPScost}
		\hat{J}(\textbf{g})=\mathbb{E}^{\textbf{g}}\left[\sum_{t=0}^{T-1} c_t(\hat{X}_t, U_t)+c_T(\hat{X}_T)\right],
	\end{align}
	in Problem 1.
\end{theorem}
\begin{proof}
	If $p(X_{t},\hat{X}_{t}~|~Y_{0:t}, U_{0:t-1})$ is known, then, for all $t=0,\ldots,T-1$, $U_{t} = g_t\big(p(X_{t},\hat{X}_{t}~|~Y_{0:t}, U_{0:t-1})\big)$ minimizes \eqref{theo:CPSmodelgo}, which implies that
	\begin{align}			
		|X_{t+1} - \hat{X}_{t+1}|^2 = 0,
		\label{Xhat}	
	\end{align}	
	for all $t=0,\ldots,T-1$. Hence $c_t(X_t,U_l) = c_t(\hat{X}_t,U_l)$ and $c_T(X_T) = c_T(\hat{X}_T)$. 
	Therefore,
	\begin{align}			
		J(\textbf{g};\hat{x}_{0:T}) &= \mathbb{E}^{\textbf{g}}\Bigg[\sum_{t=0}^{T-1} c_t(X_t,U^{1:K}_t)
		+c_T(X_T)  \Bigg]\nonumber\\
		&=\mathbb{E}^{\textbf{g}}\Bigg[\sum_{t=0}^{T-1} c_t(\hat{X}_t,U^{1:K}_t)
		+c_T(\hat{X}_T) \Bigg] =\hat{J}(\textbf{g}).
		\label{theoCPS:1c}	
	\end{align}		
\end{proof}

\section{Illustrative Example} \label{sec:4}

In this section, we present a simple example to illustrate how to derive the optimal control strategy for a system that evolves for a time horizon $T=2$ using a model of the system and separating the learning and control tasks. The initial state, $X_0$, and disturbance, $W_0,$ of the system (primitive random variables) are Gaussian random variables with zero mean, variance $1$, and covariance $0.5$. The state of the actual system is denoted  by $\hat{X}_t, ~t=0, 1,2,$ and evolves as follows:

\begin{align}\label{eq:example1}
	\hat{X}_0 &= X_0, \nonumber\\
	\hat{X}_1 &= \hat{X}_0 +U_0 + W_0 = X_0 +U_0 + W_0, \nonumber\\
	\hat{X}_2 &= \hat{X}_1 + U_1,
\end{align}
and the observation equations are
\begin{align}\label{eq:example2}
	\hat{Y}_t = \hat{X}_{t}, \quad t=0, 1,2.			
\end{align}

The control action $U_t,$ $t = 0,1,$ of the system is given by a control strategy $\textbf{g}=\{g_t;~ t=0,1\}$, $\textbf{g}\in\mathcal{G}$,
\begin{align}\label{eq:controlexample}
	U_t =g_t(\hat{Y}_{0:t}, U_{0:t-1}),
\end{align}
where  $g_t$ is the control law, which is a measurable function $g_t: (\mathcal{Y}\times \mathcal{U}_{t-1},\mathscr{Y}\otimes\mathscr{U}_{t-1})\to (\mathcal{U}_t, \mathscr{U}_t)$.
The feasible sets of decisions $\mathcal{U}_t$ at $t =0,1$ consist of $U_0 = g_0(\hat{X}_0),$ and $U_1 = g_1(\hat{X}_0, \hat{X}_1, U_0).$
The problem is to derive the optimal control strategy $\textbf{g}^*\in\mathcal{G}$ of the system in \eqref{eq:example1} that minimizes the following cost:
\begin{align}
	J(\textbf{g})&=\min_{u_0\in\mathcal{U}_0, u_1\in\mathcal{U}_1}\frac{1}{2}\mathbb{E}^{\textbf{g}}\left[(\hat{X}_2)^2 + (U_1)^2\right].
\end{align}

We consider that the evolution of the actual system in \eqref{eq:example1} is not known. However, we have the following model available to  derive the optimal strategy $\textbf{g}\in\mathcal{G}$:
\begin{align}\label{eq:example11}
	X_0 &= X_0,\nonumber\\
	X_1 &= 2 X_0 + 3 U_0 + 4 W_0, \nonumber\\
	X_2 &= 2 X_1 + 4 U_1,
\end{align}
while the observation equations are
\begin{align}\label{eq:example2a}
	Y_t = X_{t}, \quad t= 0,1,2.			
\end{align}

The difference between \eqref{eq:example1} and \eqref{eq:example11} represents a typical discrepancy that exist between a system and the system's model.
The evolution of both the actual system and system's model at $t=1$ and $t=2$ is illustrated in Fig. \ref{fig:4} and Fig. \ref{fig:5}, respectively.

\begin{figure}
	\centering
	\includegraphics[width=0.7\linewidth, keepaspectratio]{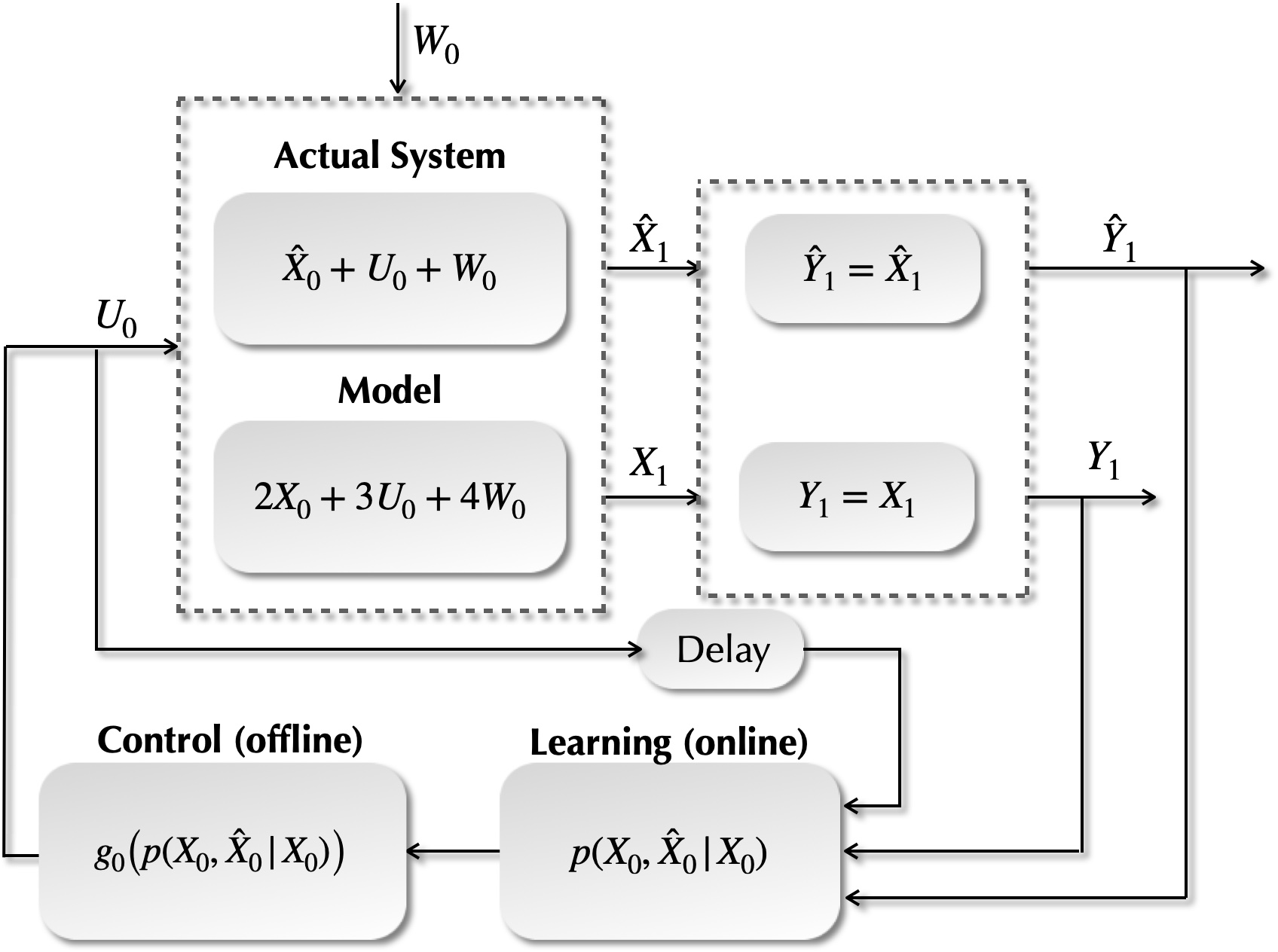} 
	\caption{The evolution of both the actual system and system's model at $t=1$.}%
	\label{fig:4}%
\end{figure}
\begin{figure}
	\centering
	\includegraphics[width=0.7\linewidth, keepaspectratio]{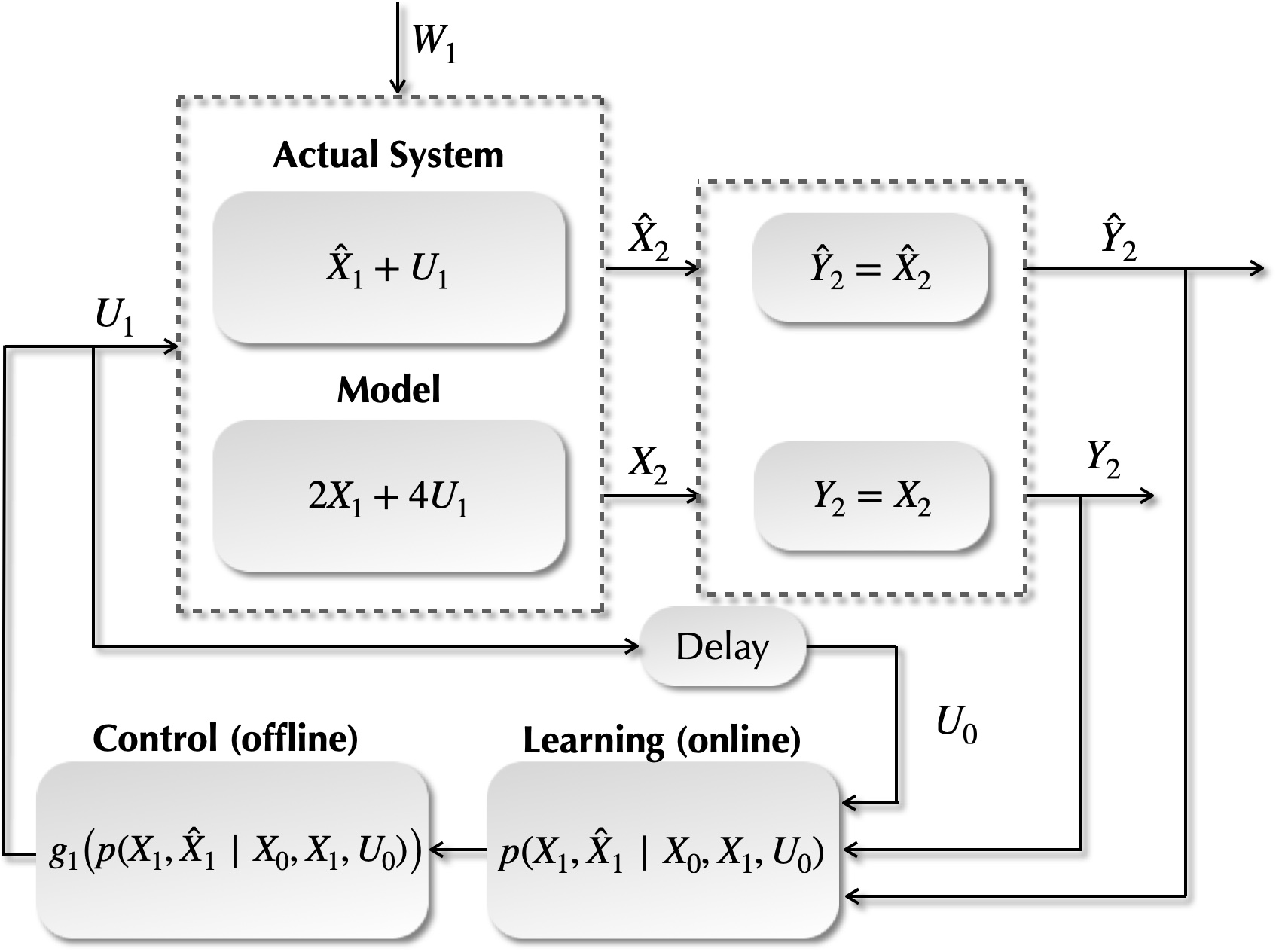} 
	\caption{The evolution of both the actual system and system's model at $t=2$.}%
	\label{fig:5}%
\end{figure}

\subsection{Optimal Control Strategy of the Actual System}
Before we proceed with the solution given by separating learning and control tasks using the system's model, we derive the optimal control strategy $\textbf{g}^*\in\mathcal{G}$ of the actual system using  \eqref{eq:example1}.
The total expected cost is
\begin{align}\label{eq:example6}
	J(\textbf{g})&=\min_{u_0\in\mathcal{U}_0, u_1\in\mathcal{U}_1}\frac{1}{2}\mathbb{E}^{\textbf{g}}\left[(\hat{X}_2)^2 + (U_1)^2\right]\nonumber\\
	&=\min_{u_0\in\mathcal{U}_0, u_1\in\mathcal{U}_1}\frac{1}{2}\mathbb{E}^{\textbf{g}}\left[(\hat{X}_1 + U_1)^2 + (U_1)^2 \right]\nonumber\\
	&=\min_{u_0\in\mathcal{U}_0, u_1\in\mathcal{U}_1}\frac{1}{2}\mathbb{E}^{\textbf{g}}\left[(X_0 + U_0 +W_0 + U_1)^2 + (U_1)^2 \right].
\end{align}

Since the primitive random variables are Gaussian with zero mean, variance $1$, and covariance $0.5,$ the problem \eqref{eq:example6} has a unique optimal solution which can be computed in a straightforward manner. The optimal solution is
\begin{align}\label{eq:example8}
	U_0 = \frac{1}{2}X_0,\quad U_1=-\frac{1}{4}X_0.
\end{align}

\subsection{Solution Given by Separating Learning and Control}
In practice, the evolution of the actual system \eqref{eq:example1} is not known. However, a model of the system is available that can be used to derive the optimal control strategy. Such  model-based control approaches cannot effectively facilitate optimal solutions with performance guarantees due to the discrepancy between the model and the actual system.

To address the problem in \eqref{eq:example6}, we apply the framework presented in Section \ref{sec:3}. More specifically, we use the model \eqref{eq:example11} that is available and seek to derive the separated control strategy $\textbf{g}\in\mathcal{G}^s$, $\textbf{g}=\{g_t;~t = 0,1\}$, where the control law is of the form $g_t\big(\mathbb{P}(X_t,\hat{X}_t~|~Y_{0:t}, U_{0:t-1})\big)$, that minimizes the following expected total cost given in Theorem \ref{theo:CPSmodel},

\begin{align}\label{eq:cost1}
	&J(\textbf{g};\hat{x}_{0:2})\nonumber\\
	&=\min_{u_0\in\mathcal{U}_0, u_1\in\mathcal{U}_1}\frac{1}{2}\mathbb{E}^{\textbf{g}}\left[(X_2)^2 + (U_1)^2 +\beta (X_1-\hat{X}_1)^2 + \beta (X_2 - \hat{X}_2)^2)~|~X_0, X_1, U_0\right].
\end{align}
From \eqref{eq:example11} and taking $\beta=1$, \eqref{eq:cost1} becomes

\begin{align}\label{eq:cost2}
	&J(\textbf{g};\hat{x}_{0:2})\nonumber\\
	&=\min_{u_0\in\mathcal{U}_0, u_1\in\mathcal{U}_1}\frac{1}{2}\mathbb{E}^{\textbf{g}}\Big[(2 X_1 + 4 U_1)^2 + (U_1)^2 + (X_1-\hat{X}_1)^2 + (X_2 - \hat{X}_2)^2)~|~X_0, X_1, U_0\Big]\nonumber\\
	&=\min_{u_0\in\mathcal{U}_0, u_1\in\mathcal{U}_1}\frac{1}{2}\mathbb{E}^{\textbf{g}}\Big[\big(2 (2 X_0 + 3 U_0 + 4 W_0) + 4 U_1\big)^2 + (U_1)^2 + (X_1-\hat{X}_1)^2 \nonumber\\
	&+ (X_2 - \hat{X}_2)^2)~|~X_0, X_1, U_0\Big].
\end{align}

To achieve the minimum  in \eqref{eq:cost2}, the control action $U_0$ and $U_1$ should make the last two terms equal to zero, namely
\begin{align}
	&\mathbb{E}^{\textbf{g}}[X_1 - \hat{X}_1] = \mathbb{E}^{\textbf{g}}[2 X_0 + 3 U_0 + 4 W_0-\hat{X}_1~|~X_0] = 0, \label{eq:cost3a}\\
	& \mathbb{E}^{\textbf{g}}[X_2 - \hat{X}_2] = \mathbb{E}^{\textbf{g}}[2 X_1 + 4 U_1 - \hat{X}_2~|~X_0, X_1, U_0] = 0. \label{eq:cost3b}
\end{align}

From \eqref{eq:cost3a}, it follows that
\begin{align}
	& \mathbb{E}^{\textbf{g}}[U_0] = \mathbb{E}^{\textbf{g}}\Big[ \frac{\hat{X}_1 - 2 X_0 - 4W_0}{3}~|~X_0\Big] = g_0\big(p(X_0,\hat{X}_0~|~X_0)\big). \label{eq:cost4a}
\end{align}

Similarly, from \eqref{eq:cost3b}, it follows that
\begin{align}
	\mathbb{E}^{\textbf{g}}[U_1] &= \mathbb{E}^{\textbf{g}}\Big[\frac{\hat{X}_2 - 4 X_0 - 6 U_0 - 8 W_0}{4}~|~X_0, X_1, U_0\Big] \nonumber\\
	&= g_1 \big(p(X_1,\hat{X}_1~|~X_0, X_1, U_0)\big). \label{eq:cost4b}
\end{align}

Thus, $U_0$ and $U_1$ in \eqref{eq:cost4a} and \eqref{eq:cost4b}, respectively, are parameterized with respect to the realizations of the state of the actual system, i.e., $\hat{x}_0= x_0$, $\hat{x}_1$ and $\hat{x}_2$, and  make the last two terms in \eqref{eq:cost2} vanish. 

Next, consider that both the actual system and system's model evolve over a time horizon $T=2$ (see Figs. \ref{fig:4} and \ref{fig:5}) using the control actions $U_0$ and $U_1$ in \eqref{eq:cost4a} and \eqref{eq:cost4b}. As we observe the realizations of $X_0$, $X_1$, and $U_0$, we learn  the information states $p(X_{0},\hat{X}_{0} ~|~ X_{0})$ and $p(X_1, \hat{X}_{1} ~|~ X_{0}, X_{1}, U_{0})$ of the system. From Proposition \ref{theo:CPSstate}, it follows that to learn $p(X_{0},\hat{X}_{0} ~|~ X_{0})$ and $p(X_1, \hat{X}_{1} ~|~ X_{0}, X_{1}, U_{0}),$ we essentially need to learn the conditional probabilities $p(X_{0} ~|~ X_{0})$, $p(X_{1} ~|~ X_{0},X_{1}, U_{0})$, $p(\hat{X}_{0} ~|~ \hat{X}_{0}, \hat{X}_{1}, U_{0})$, and $p(\hat{X}_{0},\hat{X}_{1}~|~U_{0},U_{1})$. The implication of learning the information states is that we can compute the realizations of $U_0$ and $U_1$ in \eqref{eq:cost4a} and \eqref{eq:cost4b}.

By substituting \eqref{eq:cost4a} in \eqref{eq:cost2}, we obtain
\begin{align}\label{eq:cost5}
	J(\textbf{g};\hat{x}_{0:2}) & = \min_{u_0\in\mathcal{U}_0, u_1\in\mathcal{U}_1}\frac{1}{2}\mathbb{E}^{\textbf{g}}\Big[\big(2 (2 X_0 + 3 \frac{\hat{X}_1 - 2 X_0 - 4W_0}{3} + 4 W_0) + 4 U_1\big)^2 \nonumber\\
	&+ (U_1)^2  ~|~X_0, X_1, U_0\Big] \nonumber\\
	& =  \min_{u_0\in\mathcal{U}_0, u_1\in\mathcal{U}_1}\frac{1}{2}\mathbb{E}^{\textbf{g}}\Big[\big(2 (\hat{X}_1 + 4 U_1\big)^2 + (U_1)^2  ~|~X_0, X_1, U_0\Big]\nonumber\\
	& =  \min_{u_0\in\mathcal{U}_0, u_1\in\mathcal{U}_1}\frac{1}{2}\mathbb{E}^{\textbf{g}}\Big[\big(2 (X_0 +U_0 + W_0 + 4 U_1\big)^2 + (U_1)^2  ~|~X_0, X_1, U_0\Big].
\end{align}

Next, to find the minimum in \eqref{eq:cost5} at time $t=0$, we take the partial derivative with respect to $U_0$

\begin{align}\label{eq:cost6}
	\frac {\partial \frac{1}{2}\mathbb{E}^{\textbf{g}}\Big[\big(2 (X_0 +U_0 + W_0 + 4 U_1\big)^2 + (U_1)^2  \Big]}{\partial U_0}  = \mathbb{E}^{\textbf{g}}\Big[\big(2 (X_0 +U_0 + W_0 + 4 U_1\big)  \Big] = 0.
\end{align}

Note that,  at $t=0$, $U_1$ is not taken into consideration yet, hence

\begin{align}\label{eq:cost7}
	\mathbb{E}^{\textbf{g}}\Big[\big(2 (X_0 +U_0 + W_0\big) \Big] = 0,
\end{align}

which yields the same solution $U_0 = \frac{1}{2}X_0$ as in \eqref{eq:example8}.

Next, substituting \eqref{eq:cost4b} into  the system's model $X_2 = 2 X_1 + 4 U_1$, we obtain

\begin{align}\label{eq:cost8}
	X_2 &= 2 X_1 + 4 \frac{\hat{X}_2 - 4 X_0 - 6 U_0 - 8 W_0}{4} \nonumber\\
	&= 2 (2 X_0 + 3 U_0 + 4 W_0) + \hat{X}_2 - 4 X_0 - 6 U_0 - 8 W_0\nonumber\\
	&= \hat{X}_2 ,
\end{align}
hence the expected total cost $J(\textbf{g};\hat{x}_{0:2})$ in \eqref{eq:cost1} becomes

\begin{align}\label{eq:cost9}
	J(\textbf{g};\hat{x}_{0:2}) &= \min_{u_0\in\mathcal{U}_0, u_1\in\mathcal{U}_1}\frac{1}{2}\mathbb{E}^{\textbf{g}}\Big[(\hat{X}_2)^2 + (U_1)^2 )\Big] \nonumber\\
	&= \min_{u_0\in\mathcal{U}_0, u_1\in\mathcal{U}_1}\frac{1}{2}\mathbb{E}^{\textbf{g}}\Big[(\hat{X}_1 + U_1)^2 + (U_1)^2 )\Big].
\end{align}

Next, to find the minimum in \eqref{eq:cost9} at time $t=1$, we take the partial derivative with respect to $U_1$

\begin{align}\label{eq:cost10}
	\frac {\partial \frac{1}{2}\mathbb{E}^{\textbf{g}}\Big[(\hat{X}_1 + U_1)^2 + (U_1)^2 )\Big]}{\partial U_1}  = \mathbb{E}^{\textbf{g}}\Big[\big( \hat{X}_1 + U_1 + U_1 \big)  \Big] = 0
\end{align}

or

\begin{align}\label{eq:cost11}
	\mathbb{E}^{\textbf{g}}\Big[\big( X_0 +U_0 + W_0 + 2 U_1 \big)  \Big] = 0
\end{align}

which yields the same solution $U_1=-\frac{1}{4}X_0$ as in \eqref{eq:example8}.

\section{Concluding Remarks and Discussion}\label{sec:5}
In most CPS applications, we typically use a model to derive optimal control strategies. Such model-based control approaches cannot effectively facilitate optimal solutions with performance guarantees due to the discrepancy between the model and the actual CPS. On the other hand, in most CPS there is a large volume of data which is added to the system gradually in real time and not altogether in advance. Thus, traditional supervised learning approaches cannot always facilitate robust solutions using data derived offline. By contrast, applying reinforcement learning approaches directly to the actual CPS might impose negative implications on safety and robust operation of the system. 

In this chapter, we presented a theoretical framework that  circumvents these challenges by developing data-driven  approaches at the intersection of learning and control. 
We used the actual system that we seek to optimally control online, in parallel with a  model of the system that we have available.  We established an information state which is the conditional joint probability distribution of the states of the model and the actual system at time $t$ given all data available of the model up until time $t$. Then, we used this information state in conjunction with the model to derive offline separated control strategies. Since the optimal strategies are derived offline, the state of the actual system  is not known,  and thus the optimal strategy of the model was parameterized with respect to all realizations of the  state of the actual system. However,  since the control strategy and the process of estimating the information state are separated,  we are able to learn the information state of the system online, while we operate simultaneously the model and the actual system in real time. Namely,  the optimal strategy derived for the  model offline, which is parameterized with respect to the state of the actual system, is used to operate the actual system in parallel with the model. As we operate both the actual system and the model and collect data, we can learn the information state online. We showed that when the information state becomes known online through learning, the separated control strategy of the model derived offline is optimal for the actual system. 

The framework departs from traditional model-based and supervised (or unsupervised)
learning approaches. Using separated control strategies, we can combine fundamental methods of control theory and learning aimed at facilitating optimal solutions with performance guarantees for a wide range of CPS applications such as emerging mobility systems,  mobility markets, networked control systems, communication networks, smart power grids, power systems, social media platforms, and internet of things.

\section{Acknowledgments}
This research was supported by NSF under Grants CNS-2149520 and CMMI-2219761.

\bibliographystyle{IEEEtran}
\bibliography{TAC_learn_Andreas,TAC_Ref_Andreas,TAC_Ref_IDS,TAC_Ref_structure,TAC2_learn}

\end{document}